\documentclass[11pt]{amsproc}
\usepackage{amssymb,amsthm,anysize,enumerate,float,times,amsmath,hyperref,tikz,url}
\hypersetup{citecolor=red, linkcolor=blue, colorlinks=true}
\usetikzlibrary{arrows,shapes,positioning}

\theoremstyle{plain}
\newtheorem{theorem}{Theorem}[section]
\newtheorem{corollary}[theorem]{Corollary}
\newtheorem{proposition}[theorem]{Proposition}
\newtheorem{lemma}[theorem]{Lemma}

\theoremstyle{remark}
\newtheorem{remark}[theorem]{Remark}

\newcommand{\PG}{\textnormal{PG}}

\newcommand{\GammaL}{\text{$\Gamma$L}}
\newcommand{\F}{\mathbb{F}}

\newcommand{\GL}{\mathrm{GL}}
\newcommand{\SL}{\mathrm{SL}}
\newcommand{\PGL}{\mathrm{PGL}}

\newcommand{\PSL}{\mathrm{PSL}}

\newcommand{\Sp}{\mathrm{Sp}}

\renewcommand\le{\leqslant}
\renewcommand\ge{\geqslant}

\title[Generalised quadrangles and transitive pseudo-hyperovals]{Generalised quadrangles and transitive pseudo-hyperovals}
\author{John Bamberg, S. P. Glasby, Tomasz Popiel and Cheryl E. Praeger}

\address{
Centre for the Mathematics of Symmetry and Computation\\
School of Mathematics and Statistics\\
The University of Western Australia\\
35 Stirling Highway, Crawley, W.A. 6009, Australia.\newline
Email: \texttt{\{john.bamberg, stephen.glasby$^*$, tomasz.popiel, cheryl.praeger$^\dag$\}@uwa.edu.au}\\
\newline $^*$Also affiliated with The
Department of Mathematics, University of Canberra, Australia. 
\newline $^\dag$ Also affiliated with King Abdulaziz University, Jeddah, Saudi Arabia.
}



\begin{document}

\begin{abstract} 
A {\em pseudo-hyperoval} of a projective space $\PG(3n-1,q)$, $q$ even, is a set of $q^n+2$ subspaces of dimension $n-1$ such that any three span the whole space.   
We prove that a pseudo-hyperoval with an irreducible transitive stabiliser is elementary. 
We then deduce from this result a classification of the thick generalised quadrangles $\mathcal{Q}$ that admit a point-primitive, line-transitive automorphism group with a point-regular abelian normal subgroup. 
Specifically, we show that $\mathcal{Q}$ is flag-transitive and isomorphic to $T_2^*(\mathcal{H})$, where $\mathcal{H}$ is either the regular hyperoval of $\PG(2,4)$ or the Lunelli--Sce hyperoval of $\PG(2,16)$. 
\end{abstract}

\maketitle

\section{Introduction and results} \label{intro}

A {\em generalised quadrangle} $\mathcal{Q}$ of order $(s,t)$ is a point--line geometry such that each line (respectively point) is incident with exactly $s+1$ points (respectively $t+1$ lines), and such that the bipartite incidence graph of $\mathcal{Q}$ has diameter $4$ and girth $8$. 
Generalised quadrangles, together with the other {\em generalised polygons}, were introduced by Tits~\cite{Tits} in an attempt to find a systematic geometric interpretation for the simple groups of Lie type. 
Their importance arises from this connection, and is harnessed through the study of their automorphism groups. 
A topic of particular interest has been generalised quadrangles that admit a {\em point-regular} group of automorphisms, namely an automorphism group $N$ with the property that, for each pair of points $x$, $x'$, exactly one element of $N$ maps $x$ to $x'$. 
The study of such generalised quadrangles was initiated by Ghinelli~\cite{Ghinelli} in connection with the theory of difference sets. 
Ghinelli conjectured that a generalised quadrangle with $s=t$ cannot admit a point-regular automorphism group, and provided evidence for this conjecture by proving that a Frobenius group or a group with a nontrivial centre cannot act regularly on the points of such a generalised quadrangle if $s$ is even. 
Yoshiara~\cite{Yoshiara} proved that a generalised quadrangle with $s=t^2$ does not admit a point-regular automorphism group. 
Bamberg and Giudici~\cite{BambergGiudici} studied point-regular automorphism groups of some known generalised quadrangles, and in particular determined all groups that act point-regularly on the thick {\em classical} generalised quadrangles. 
Here {\em thick} means that $s \ge 2$ and $t \ge 2$, and this assumption excludes certain trivial examples. 

As in \cite{PenttilaVandeVoorde}, a {\em pseudo-hyperoval} is a set of $(n-1)$-dimensional subspaces of a projective space $\PG(3n-1,q)$, such that the set has cardinality $q^n+2$ and any three elements span the whole space. 
We note that pseudo-hyperovals are also called {\em generalized hyperovals} \cite{CoopThas}. 
For $n=1$, a pseudo-hyperoval is just a {\em hyperoval} of $\PG(2,q)$, namely a set of $q+2$ points such that no three lie on a line. 
De Winter and K.~Thas \cite{DeWinterThasRegular} showed that if a thick generalised quadrangle $\mathcal{Q}$ admits an {\em abelian} point-regular automorphism group $N$, then $\mathcal{Q}$ arises from a pseudo-hyperoval in the following sense: 
$N$ turns out to be the additive group of a vector space (and hence elementary abelian), and the setwise stabilisers in $N$ of the lines incident with a fixed point of $\mathcal{Q}$ are subspaces forming a pseudo-hyperoval (when $N$ is viewed projectively). 
Conversely, as explained in Section~\ref{corProof}, every pseudo-hyperoval gives rise to a generalised quadrangle by means of a geometric construction known as {\em generalised linear representation}. 

The aim of the present paper is to classify a family of pseudo-hyperovals that admit additional symmetry, and to thereby obtain a characterisation of the corresponding generalised quadrangles.
We begin by considering pseudo-hyperovals that admit a transitive automorphism group. 
Our first result says that for a thick generalised quadrangle $\mathcal{Q}$ with a point-regular abelian automorphism group, this condition is equivalent to transitivity on the flags (incident point--line pairs) of $\mathcal{Q}$.

\begin{theorem}\label{transitivepseudohyperoval}
Let $\mathcal{Q}$ be a thick generalised quadrangle of order $(s,t)$ admitting an automorphism group $H$ with an abelian normal subgroup $N$ that acts regularly on the points of $\mathcal{Q}$. 
Choose any point $x$ of $\mathcal{Q}$, let $\ell_1,\ldots,\ell_{t+1}$ be the lines incident with $x$, and, for each $i\in\{1,\ldots,t+1\}$, let $U_i := N_{\ell_i}$ be the setwise stabiliser of $\ell_i$ in $N$. 
Then the following are equivalent:
\begin{enumerate}[{\rm (i)}]
\item $H$ acts transitively on the lines of $\mathcal{Q}$;
\item $H$ acts transitively on the flags of $\mathcal{Q}$; 
\item $H$ acts transitively on the pseudo-hyperoval $\{U_1,U_2,\ldots,U_{t+1}\}$, by conjugation. 
\end{enumerate}
\end{theorem}

Theorem~\ref{transitivepseudohyperoval} is proved in Section~\ref{prelim}. 
It implies that a classification of transitive pseudo-hyperovals would yield a classification of the generalised quadrangles that admit a line-transitive automorphism group with a point-regular abelian normal subgroup, and, moreover, that such generalised quadrangles are, in fact, flag-transitive. 
By a result of J.~A. Thas \cite[\S4.5]{Thas71}, if a projective space $\PG(3n-1,q)$ contains a pseudo-hyperoval, then $q=2^f$ for some positive integer $f$. 
For small values of the product $nf$, we appeal to some existing results to classify the transitive pseudo-hyperovals of $\PG(3n-1,2^f)$, beginning with Korchmaros'~\cite{Korch} classification of transitive hyperovals of $\PG(2,q)$. 
We obtain the following result. 

\begin{proposition} \label{propn:elementary}
Suppose that $nf\le 4$. Then every pseudo-hyperoval of $\PG(3n-1,2^f)$ is elementary. 
Moreover, every transitive pseudo-hyperoval of $\PG(3n - 1, 2^f)$ is the field-reduced image of a transitive hyperoval of $\PG(2,2^{nf})$, namely of either
\begin{enumerate}[{\rm (i)}]
\item the regular hyperoval of $\PG(2,2)$ or $\PG(2,4)$, or 
\item the Lunelli--Sce hyperoval of $\PG(2,16)$.
\end{enumerate}
\end{proposition}

Here {\em elementary} means that the pseudo-hyperoval of $\PG(3n-1,2^f)$ is the image of a hyperoval of $\PG(2,2^{nf})$ under field reduction, as explained in the first paragraph of Section~\ref{sec3}, and a {\em regular hyperoval} (or {\em hyperconic}) is the completion of a conic to a hyperoval by addition of its nucleus. 
For details on the Lunelli--Sce hyperoval, we refer the reader to the article by Brown and Cherowitzo \cite{BrownCherowitzo}. 

Proposition~\ref{propn:elementary} is proved in Section~\ref{sec3}.
For $nf>4$, we instead investigate the subgroups of $\GL(3n,2^f)$ that act transitively on a set of $2^{nf}+2$ subspaces of dimension $n$. 
As explained in Section~\ref{sec4}, the order of such a subgroup must be divisible by the largest primitive divisor of $2^e-1$, where $e=2d/3-2$, namely the largest divisor of $2^e-1$ that does not divide $2^i-1$ for any $i<e$. 
The subgroups with this property can be determined from a result of Bamberg and Penttila~\cite[Theorem 3.1]{Bamberg:2008rr} by considering each of Aschbacher's \cite{Aschbacher} geometric classes of maximal subgroups of the general linear group.  
However, computational experiments suggest that there may be a very large number of examples in the case where the subgroup fixes a proper subspace of the underlying vector space:

\begin{remark} \label{eg}
{\sf GAP} \cite{gap} computations show that the pseudo-hyperoval of $\PG(11,2)$ obtained by field reduction of the Lunelli--Sce hyperoval of $\PG(2,16)$ has stabiliser with $14$ transitive subgroups, with half of these subgroups fixing a $5$-dimensional projective subspace. 
Details are given in Appendix~\ref{app}. 
\end{remark}

We therefore narrow our search to pseudo-hyperovals with an {\em irreducible} transitive stabiliser. 
We prove in Section~\ref{sec4} that for $nf>4$ there exists no irreducible subgroup of $\GL(3n,2^f)$ that acts transitively of degree $2^{nf}+2$, thereby proving that for $nf>4$ there is no pseudo-hyperoval of $\PG(3n-1,2^f)$ with an irreducible transitive stabiliser. Together with Proposition~\ref{propn:elementary} , this gives our main theorem:

\begin{theorem} \label{mainthm}
Suppose that $\PG(3n-1,2^f)$ contains a pseudo-hyperoval $\mathcal{O}$ admitting an irreducible transitive group of automorphisms. Then $nf\le 4$ and $\mathcal{O}$ is the field-reduced image of either
\begin{enumerate}[{\rm (i)}]
\item the regular hyperoval of $\PG(2,2)$ or $\PG(2,4)$, or 
\item the Lunelli--Sce hyperoval of $\PG(2,16)$.
\end{enumerate}
\end{theorem}

Our final result is a corollary to Theorem~\ref{mainthm}, classifying the corresponding generalised quadrangles. 
As explained in Section~\ref{corProof}, for $\mathcal{Q}$ as in Theorem~\ref{transitivepseudohyperoval}, the condition that the stabiliser of the pseudo-hyperoval $\{ U_1,\ldots,U_{t+1} \}$ is irreducible is equivalent to the condition that $H$ acts {\em primitively} on the points of $\mathcal{Q}$, namely that it preserves no nontrivial partition of the points. 
Point-primitivity is a natural assumption to make in the context of the existing literature, notably recent papers of Schneider and Van Maldeghem~\cite{SvM} and Bamberg et~al.~\cite{BGMRS} which classify point-primitive, line-primitive, flag-transitive generalised polygons. 
If $\PG(2,q)$, $q$ even, contains a hyperoval $\mathcal{H}$, then one can construct from $\mathcal{H}$ a generalised quadrangle $T_2^*(\mathcal{H})$ of order $(q-1,q+1)$ by embedding $\PG(2,q)$ as a hyperplane $\Pi$ of $\PG(3,q)$ and declaring the `points' to be the affine points (those not in $\Pi$) and the `lines' to be the affine lines meeting $\Pi$ in an element of $\mathcal{H}$, with natural incidence \cite{Hall71}. 
We prove the following in Section~\ref{corProof}.

\begin{corollary} \label{mainCorollary}
Let $\mathcal{Q}$ be a thick generalised quadrangle admitting an automorphism group that is point-primitive, line-transitive and has a point-regular abelian normal subgroup. 
Then $\mathcal{Q}$ is isomorphic to $T_2^*(\mathcal{H})$, where $\mathcal{H}$ is either
\begin{enumerate}[{\rm (i)}]
\item the regular hyperoval of $\PG(2,4)$, or
\item the Lunelli--Sce hyperoval of $\PG(2,16)$.
\end{enumerate}
\end{corollary}

Note that the regular hyperoval of $\PG(2,2)$ does not appear in Corollary~\ref{mainCorollary} because it does not yield a {\em thick} generalised quadrangle (see Section~\ref{corProof}). 
Note also that if $\mathcal{H}$ is the regular hyperoval of $\PG(2,4)$, then $T_2^*(\mathcal{H})$ is the unique generalised quadrangle of order $(3,5)$ \cite[6.2.4]{FGQ}.

\section{Proof of Theorem~\ref{transitivepseudohyperoval}} \label{prelim}

Let $\mathcal{Q}$ be a thick generalised quadrangle of order $(s,t)$. 
Then $\mathcal{Q}$ is a partial linear space (a point--line incidence structure such that any line is incident with at least two points and any two distinct points are incident with at most one line) satisfying the so-called \emph{generalised quadrangle axiom}:
given a point $x$ and line $\ell$ not incident with $x$, there is a unique line incident with $x$ that is concurrent with $\ell$.
The number of points of $\mathcal{Q}$ is $(s+1)(st+1)$, and the number of lines of $\mathcal{Q}$ is $(t+1)(st+1)$ \cite[1.2.1]{FGQ}. 

De Winter and K. Thas \cite{DeWinterThasRegular} showed that if $\mathcal{Q}$ admits a point-regular abelian group of automorphisms $N$, then $N$ is elementary abelian and $\mathcal{Q}$ is isomorphic to a generalised quadrangle arising from a pseudo-hyperoval. 
We give details as to what this conclusion means by retracing their steps, including some proofs for ease of reference. 
We first observe that \cite[Theorem 2.3]{DeWinterThasRegular} generalises as follows.

\begin{proposition} \label{order} 
Let $\mathcal{Q}$ be a thick generalised quadrangle of order $(s,t)$ admitting a point-regular automorphism group $N$ such that $|N|=p^d$ for some prime $p$ and some positive integer $d$. 
Then $d$ is divisible by $3$ and $(s,t) = (p^{d/3}-1,p^{d/3}+1)$.
\end{proposition}

\begin{proof}
Let $\mathcal{P}$ denote the point set of $\mathcal{Q}$. 
Then $|N| = |\mathcal{P}|$, namely $p^d = (s+1)(st+1)$. 
In particular, $s+1$ divides $p^d$, so $s+1=p^k$ for some $k$ and hence $t=(p^{d-k}-1)/(p^k-1)$. 
Since $t\ge 2$, $p^{d-2k} = p^{d-k}/p^k = (st+1)/(s+1) > 1$, and so $k<d/2$. 
On the other hand, $|\mathcal{P}| < (s+1)^4$ \cite[Lemma~2.5(ii)]{BGMRS}, so $p^d < p^{4k}$ and hence $k>d/4$. 
Moreover, $k$ divides $d$ because $t = (p^{d-k}-1)/(p^k-1)$ is an integer. 
Therefore, $k=d/3$, and hence $s=p^{d/3}-1$ and $t = (p^{2d/3}-1)/(p^{d/3}-1) = p^{d/3}+1$. 
\end{proof}

We now consider the situation where $\mathcal{Q}$ admits a point-regular abelian group of automorphisms $N$. 
Let $\mathrm{End}(N)$ denote the endomorphism ring of the abelian group $N$. 
Regarding $N$ as an additive group, the sum $f+g$ of $f,g \in \operatorname{End}(N)$ is defined by $(f+g)(a):=f(a)+g(a)$ for all $a\in N$, and the product $fg$ is given by function composition. 
We also make careful mention that we think of a vector space as a triple $(V,F,\varphi)$, where $V$ is an abelian group, $F$ is a field, and $\varphi$ is a ring homomorphism from $F$ into $\mathrm{End}(V)$ (that is,  `scalar multiplication' by elements of $F$). 
The following result of \cite{DeWinterThasRegular} asserts that $N$ can be identified with a projective space containing a pseudo-hyperoval. 
We include parts of the proof. 

\begin{proposition} \label{pseudohyperoval}
Let $\mathcal{Q}$ be a thick generalised quadrangle of order $(s,t)$ admitting a point-regular abelian automorphism group $N$. 
Choose a point $x$ of $\mathcal{Q}$, let $\ell_1,\ldots,\ell_{t+1}$ be the lines incident with $x$, and, for each $i\in\{1,\ldots,t+1\}$, let $U_i := N_{\ell_i}$ be the setwise stabiliser of $\ell_i$ in $N$. 
Then 
\[
K := \{ g \in \operatorname{End}(N) \mid U_i^g=U_i \text{ for all } i \in \{1,\ldots,t+1\} \}
\]
is a field, $N$ is a $K$-vector space of dimension $3n$ for some positive integer $n$, and $U_1,\ldots,U_{t+1}$ are $n$-dimensional $K$-vector subspaces of $N$ such that $N=U_i+U_j+U_k$ for every three distinct $i$, $j$ and $k$. 
\end{proposition}

\begin{proof}
The abelian group $N$ is an $\operatorname{End}(N)$-module under the natural action, with the operations described above. 
By definition, $K$ is a subring of $\operatorname{End}(N)$ and so $N$ is a $K$-module. 
By \cite[Theorem 2.4]{DeWinterThasRegular}, $K$ is a field, and so $N$ is a $K$-vector space. 
Each subgroup $U_i$ is left invariant under scalar multiplication by elements of $K$ (by definition of $K$), and thus the $U_i$ are $K$-vector subspaces of $N$.
Write $|N|=p^d$ as in Proposition~\ref{order}. 
By \cite[Lemma 2.1]{DeWinterThasRegular}, the $U_i$ have cardinality $s+1$, and by Proposition~\ref{order}, $s+1=p^{d/3}$. 
Therefore, writing $|K|=p^f$ for some $f$, we have $d/3 = nf$ for some $n$. 
That is, $N$ has dimension $d/f=3n$ as a $K$-vector space, and the $U_i$ have dimension $d/(3f)=n$. 
Since $N$ is abelian, $U_i+U_j$ is a subgroup of $N$. 
By the generalised quadrangle axiom, $\ell_i$ and $\ell_j$ determine
a {\em grid} $\Gamma$, a sub-generalised quadrangle of order $(s,1)$ described as follows and illustrated in Figure~\ref{Fig1}.
The orbit $\ell_i^{U_j}$ consists of $s+1$ disjoint lines (one for each point on $\ell_j$), and $\ell_j^{U_i}$ is a set of $s+1$ disjoint lines transverse to $\ell_i^{U_j}$. 
The underlying set of points is thus the orbit of $x$ under $U_i+U_j$, and it is not difficult to see that
 $U_i+U_j$ is the setwise stabiliser of $\Gamma$. 
Now, $\ell_k$ is another line on the point $x$ and cannot belong to the lines of $\Gamma$ (since $\Gamma$ has just two lines on every point). 
An automorphism $\theta\in U_i+U_j$ that stabilises $\Gamma$ and the line $\ell_k$ must be trivial, as we see in the following.
Recall that $N$ acts regularly on the points of $\mathcal{Q}$. 
Now, $x^\theta$ must be a point of $\Gamma$, but also must be a point incident with $\ell_k$. 
Hence $x^\theta=x$, which implies that $\theta$ is trivial. Therefore, $(U_i+U_j)  \cap U_k$ is trivial. 
By orders, we see that $N=U_i+U_j+U_k$. 
\end{proof}

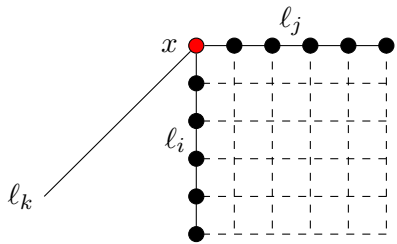
\begin{figure}[!t]
\caption{A grid stabilised by $U_i+U_j$.}
\label{Fig1}
\begin{tikzpicture}[scale =0.5]
\node (P) at (0,5) [circle, draw, fill=red!100, inner sep=2pt, minimum width=5pt, label=left:$x$] {}; 
\coordinate (A) at (0,0);
\coordinate (B) at (5,5);
\coordinate[label=above:$\ell_j$]  (C) at (2.5,5);
\coordinate[label=left:$\ell_i$]  (D) at (0,2.5);
\coordinate[label=left:$\ell_k$]  (E) at (-4,1);

\coordinate (A1) at (0,1);
\coordinate (A2) at (0,2);
\coordinate (A3) at (0,3);
\coordinate (A4) at (0,4);
\coordinate (aA1) at (5,1);
\coordinate (aA2) at (5,2);
\coordinate (aA3) at (5,3);
\coordinate (aA4) at (5,4);
\coordinate (aA5) at (5,0);
\coordinate (B1) at (1,0);
\coordinate (B2) at (2,0);
\coordinate (B3) at (3,0);
\coordinate (B4) at (4,0);
\coordinate (B5) at (5,0);
\coordinate (aB1) at (1,5);
\coordinate (aB2) at (2,5);
\coordinate (aB3) at (3,5);
\coordinate (aB4) at (4,5);
\coordinate (aB5) at (5,5);
\draw (P) -- (A);
\draw (P) -- (B);
\draw (P) -- (E);
\draw[dashed] (aA1) -- (A1);
\draw[dashed] (aA2) -- (A2);
\draw[dashed] (aA3) -- (A3);
\draw[dashed] (aA4) -- (A4);
\draw[dashed] (aA5) -- (A);

\draw[dashed] (aB1) -- (B1);
\draw[dashed] (aB2) -- (B2);
\draw[dashed] (aB3) -- (B3);
\draw[dashed] (aB4) -- (B4);
\draw[dashed] (aB5) -- (B5);

\node at  (1,5) [circle, draw, fill=black!100, inner sep=2pt, minimum width=5pt] {}; 
\node at  (2,5) [circle, draw, fill=black!100, inner sep=2pt, minimum width=5pt] {}; 
\node at  (3,5) [circle, draw, fill=black!100, inner sep=2pt, minimum width=5pt] {}; 
\node at  (4,5) [circle, draw, fill=black!100, inner sep=2pt, minimum width=5pt] {}; 
\node at  (5,5) [circle, draw, fill=black!100, inner sep=2pt, minimum width=5pt] {}; 
\node at  (0,1) [circle, draw, fill=black!100, inner sep=2pt, minimum width=5pt] {}; 
\node at  (0,2) [circle, draw, fill=black!100, inner sep=2pt, minimum width=5pt] {}; 
\node at  (0,3) [circle, draw, fill=black!100, inner sep=2pt, minimum width=5pt] {}; 
\node at  (0,4) [circle, draw, fill=black!100, inner sep=2pt, minimum width=5pt] {}; 
\node at  (0,0) [circle, draw, fill=black!100, inner sep=2pt, minimum width=5pt] {}; 
\end{tikzpicture}
\end{figure}

\begin{corollary} \label{pseudohyperovalCorr}
With assumptions as in Proposition~$\ref{pseudohyperoval}$, the set of $1$-dimensional subspaces of $N$ can be identified with $\PG(3n-1,2^f)$, and the set $\{ U_1,\ldots,U_{t+1} \}$ is a pseudo-hyperoval. 
\end{corollary}

\begin{proof}
As noted in the introduction, the existence of a pseudo-hyperoval implies that $p=2$.
\end{proof}

Before proving Theorem~\ref{transitivepseudohyperoval}, we need the following corollary to \cite[Proposition 84.1]{SnapperTroyer}.

\begin{lemma} \label{GactsonN}
Let $V$ be a vector space and suppose that $G$ is a group of automorphisms of the additive group of $V$. 
If $G$ preserves the $1$-dimensional subspaces of $V$, then $G$ is a subgroup of the group $\mathrm{\Gamma L}(V)$ of nonsingular semilinear maps on $V$.
\end{lemma}

Now consider the situation outlined by Theorem~\ref{transitivepseudohyperoval}. 
Then $H=NH_x$ and $N\cap H_x$ is trivial. 
The conjugation action of $H$ on $N$ gives a homomorphism $\psi:H\to \mathrm{Aut}(N)$, and $N\le \ker(\psi)$ as $N$ is abelian.
Hence $\psi(H_x)=\psi(H)$, and so we may identify $H$ with the semidirect product $N\rtimes G$, where $G$ is the subgroup of $\operatorname{Aut}(N)$ induced by the conjugation action of $H$ (or indeed $H_x$). Therefore, $G$ permutes the $U_i$. 
Conversely, the \emph{holomorph} $\mathrm{Hol}(N) = N\rtimes \mathrm{Aut}(N)$ of $N$ has a natural action on $N$, where $N$ acts 
regularly by right multiplication (and $\mathrm{Aut}(N)$ has the obvious action), and we can recover the subgroup $\mathrm{Aut}(N)$ by considering the stabiliser in $\mathrm{Hol}(N)$ of the trivial element of $N$.

Now let $k \in K^\times$, $y\in N$, and $g\in G$. 
Given $i\in\{1,\ldots,t+1\}$, there exists $j$ such that $U_i^{g^{-1}}=U_j$, and so
\[
U_i^{g^{-1}kg}=(U_j^k)^g=U_j^g=U_i.
\] 
Therefore, $g^{-1} k g \in K$. Write $k\cdot y$ for the action of the scalar $k$ on the element $y$.
Note that if $k$ is nonzero then it is an automorphism of $N$, and so $k\cdot y = y^k$.
We have 
\[
(k\cdot y)^g = (y^k)^g=(y^g)^{g^{-1} k g}=(k^g)\cdot(y^g).
\] 
Since also $(0\cdot y)^g = 0\cdot y^g = 0$, it follows that the $1$-dimensional $K$-vector subspace $\langle y \rangle_{K}^g$ is equal to $\langle y^g \rangle_{K}$, and hence $G$ permutes $1$-dimensional $K$-subspaces of $N$.
Thus $G$ preserves a $K$-vector space structure on $N$. 

\begin{proof}[Proof of Theorem~\ref{transitivepseudohyperoval}]
Clearly (ii) implies (i), and (iii) is equivalent to (ii), so it remains to show that (i) implies (ii). 
If $H$ is transitive on the line set $\mathcal{L}$ of $\mathcal{Q}$ then the size of the orbits $\ell^N$, where $\ell \in \mathcal{L}$, of $N$ on $\mathcal{L}$ is constant (because $N$ is normal in $H$). 
Therefore, the line stabilisers $N_\ell$ have constant size, namely the size of the $U_i$. 
That is, $|N_\ell| = s+1$, and hence $N_\ell$ acts transitively on the points of $\ell$. 
Thus $H$ is flag-transitive. 
\end{proof}

We compare Theorem~\ref{transitivepseudohyperoval} to an elegant result of Bayens and Pentilla \cite{BayensPenttila} which says the elation Laguerre plane $L(\mathcal{O})$ arising from a dual pseudo-oval $\mathcal{O}$ is flag-transitive if and only if $\mathcal{O}$ is transitive. 
We also remark that it would suffice to assume that $N$ is transitive: a transitive faithful abelian group must act regularly, and $N$, indeed $H$, acts faithfully on points as it is a subgroup of $\operatorname{Aut}(\mathcal{Q})$ (if an element of $H$ fixed each point then it would fix each line and so be the identity automorphism of $\mathcal{Q}$).

\section{Proof of Proposition~\ref{propn:elementary}} \label{sec3}

We recall some definitions before proving Proposition~\ref{propn:elementary}. 
First, let $E$ be a degree $b$ field extension of a finite field $F$. 
The $d$-dimensional vector space $V_d(F)$ over $F$ can be thought of as a $d/b$-dimensional vector space over $E$ if one identifies the $d/b$-vectors of $F^{b}$ with the $d$-vectors of $F$.
Hence, the linear transformations of $V_{d/b}(E)$ induce linear transformations of $V_d(F)$ (but not conversely), and we obtain an embedding $\mathrm{GL}_{d/b}(E)\le \mathrm{GL}_d(F)$.
This induces a map $\varphi$ from subspaces of $\PG(d/b-1,E)$ to subspaces of $\PG(d-1,F)$, with the following properties \cite[Lemma~2.2]{fieldreduction}:
(i) $\varphi$ is injective; 
(ii) $\varphi$ maps $(t-1)$-dimensional (projective) subspaces to $(bt-1)$-dimensional subspaces; 
(iii) $\varphi$ is incidence preserving; and 
(iv) any two distinct elements of the image of $\varphi$ are disjoint. 
In particular, if $d=3b$ then the {\em field-reduced image} $\varphi(\mathcal{H})$ in $\PG(3b-1,q)$ of a hyperoval $\mathcal{H}$ in $\PG(2,q^b)$ is a pseudo-hyperoval. 
Such a pseudo-hyperoval is said to be {\em elementary}. 

Corollary~\ref{mainCorollary} refers to a construction of a generalised quadrangle $T_2^*(\mathcal{H})$ from a hyperoval $\mathcal{H}$. 
We now recall a more general construction, known as {\em generalised linear representation}. 
Let $\mathcal{S}$ be a nonempty set of disjoint $m$-dimensional subspaces of $\PG(r,q)$, $r\ge 2$, and embed $\PG(r,q)$ as a hyperplane $\Pi$ in $\PG(r+1, q)$. 
The {\em generalised linear representation} $T_{r,m}^*(\mathcal{S})$ of $\mathcal{S}$ is the point--line incidence structure with `points' the points of $\PG(r+1, q)$ not in $\Pi$, `lines' the $(m+1)$-subspaces of $\PG(r+1, q)$, not in $\Pi$, that are incident with some element of $\mathcal{S}$, and natural incidence \cite[Section 3.2]{DWRVdV}. 
Observe that $T_{2,0}^*$ is the same as $T_2^*$. 
We can also view the affine structure $\PG(r+1,q)\setminus \PG(r,q)$
directly within the underlying vector space $V(r+1,q)$. 
The hyperplane at infinity is regarded as the set of points with homogeneous coordinates satisfying $x_{r+2}=0$, and the points not in this hyperplane have $x_{r+2}=1$. 
Truncating to the first $r+1$ coordinates yields a bijection between $V(r+1,q)$ and $\PG(r+1,q)\setminus \PG(r,q)$. 
It is then straightforward to show that $T_{r,m}^*(\mathcal{S})$ is isomorphic to the point--line incidence structure with `points' the vectors of $V(r+1,q)$, `lines' the right cosets of the $(m+1)$-dimensional vector subspaces given by $\mathcal{S}$, and incidence being natural inclusion. 

Now, $T_{r,m}^*(\mathcal{S})$ is, in general, a partial linear space of order $(q^{m+1}-1,|\mathcal{S}|-1)$. 
If $\mathcal{S}$ is a pseudo-hyperoval of $\PG(3b-1,q)$ then $T^*_{3b-1,b-1}(\mathcal{S})$ is a generalised quadrangle. 
The following result, which can be found in \cite[Stelling 1.7.5]{Vanhove}, implies, in particular, that a hyperoval $\mathcal{H}$ and the pseudo-hyperoval obtained from $\mathcal{H}$ by field reduction yield isomorphic generalised quadrangles.

\begin{lemma} \label{lem:Tstar}
Let $\mathcal{S}$ be a set of disjoint $m$-dimensional subspaces of $\PG(r,q^b)$, and let $\mathcal{S}'$ be the image of $\mathcal{S}$ under field reduction to $\PG((r+1)b-1,q)$. Then
\[
T^*_{r,m}(\mathcal{S})\cong T^*_{r',m'}(\mathcal{S}'), \quad \text{where } (r',m')=((r+1)b-1,(m+1)b-1).
\]
\end{lemma}

\begin{corollary} \label{cor:Tstar}
Let $\mathcal{H} \subset \PG(2,q^b)$ be a hyperoval and $\mathcal{O} \subset \PG(3b-1,q)$ the pseudo-hyperoval obtained from $\mathcal{H}$ via field reduction. 
Then the generalised quadrangles $T_2^*(\mathcal{H})$ and $T_{3b-1,b-1}^*(\mathcal{O})$ are isomorphic.
\end{corollary}

\begin{proof}[Proof of Proposition~\ref{propn:elementary}]
Recall first the result of Korchmaros~\cite{Korch} which says that if $\mathcal{H}$ is a transitive hyperoval of $\PG(2,q)$ then either (i) $q\in\{2,4\}$ and $\mathcal{H}$ is the regular hyperoval, or (ii) $q=16$ and $\mathcal{H}$ is the Lunelli--Sce hyperoval. 
In particular, taking $q=2$ proves the proposition for $nf=1$. 

Now suppose that $2 \le nf \le 4$. 
The field reduction $\PG(3n-1,2^f) \to \PG(3nf-1,2)$ maps pseudo-hyperovals to pseudo-hyperovals, so to prove that every pseudo-hyperoval of $\PG(3n-1,2^f)$ is elementary, it suffices to prove that every pseudo-hyperoval of $\PG(3nf-1,2)$ is elementary. 
If we remove an element from a pseudo-hyperoval of $\PG(3nf-1,2)$, we obtain a pseudo-oval. 
Conversely, by a theorem of J. A. Thas \cite[\S4.10]{Thas71}, every pseudo-oval of $\PG(3nf-1,2)$ extends to a unique pseudo-hyperoval.
The pseudo-ovals of $\PG(3nf-1,2)$ are known, and we can therefore show that the pseudo-hyperovals are all elementary and determine which 
have a transitive stabiliser. 

Consider first the case $nf=2$. 
As explained in Section~4 of Penttila's unpublished manuscript~\cite{Penttila}, there is a unique pseudo-oval of $\PG(5,2)$.
Thus, by Thas' result \cite{Thas71}, $\PG(5,2)$ contains a unique pseudo-hyperoval: the pseudo-hyperconic, that is, the field-reduced image of the hyperconic (regular hyperoval) of $\PG(2,4)$. 
By Korchmaros' result \cite{Korch}, the hyperconic of $\PG(2,4)$ has a transitive stabiliser, and so the pseudo-hyperconic also has a transitive stabiliser. 

Now let $nf=3$. 
By \cite[Theorem~3]{Penttila}, there are precisely two pseudo-ovals of $\PG(8,2)$ and both are elementary: a pseudo-conic and a pseudo-pointed conic. 
Both of these pseudo-ovals complete to the same pseudo-hyperoval, namely the pseudo-hyperconic, and hence by Thas' result, this is the only pseudo-hyperoval of $\PG(8,2)$. 
By Korchmaros' result, the hyperconic $\mathcal{H}$ of $\PG(2,8)$ is intransitive, and we must show that the pseudo-hyperconic $\mathcal{O}$ of $\PG(8,2)$ is also intransitive. 
The stabiliser in $\PGL(2,8)$ of $\mathcal{H}$ contains $\PSL(2,8)$ and has precisely two orbits: a conic (9 points) and its nucleus (one point). 
Moreover, $\PSL(2,8)$ acts $3$-transitively on the conic.
If the stabiliser $G \le \PGL(9,2)$ of $\mathcal{O}$ acted transitively on the pseudo-hyperconic, then it would act $4$-transitively, and hence primitively.
The only $4$-transitive groups of degree $10$ are $A_{10}$ and $S_{10}$ (see \cite{Sims}), and we claim that $G$ cannot induce $A_{10}$ (nor $S_{10}$, therefore) on $\mathcal{O}$. 
Now, a collineation of $\PG(8,2)$ that stabilises $\mathcal{O}$ induces an automorphism of the generalised quadrangle $T_{8,2}^*(\mathcal{O})$, and by Corollary~\ref{cor:Tstar}, $T_{8,2}^*(\mathcal{O}) \cong T_2^*(\mathcal{H})$. 
Hence there is a homomorphism from $G$ to the automorphism group of $T_2^*(\mathcal{H})$, and $A_{10}$, being a simple group, must appear in the image of this homomorphism. 
However, $T_2^*(\mathcal{H})$ is isomorphic to the Payne derivation of the generalised quadrangle $W(3,8)$ (see \cite[3.2.6]{FGQ}), and Grunh\"ofer et~al.~\cite{Grundhofer} have shown that the automorphism group of the Payne derivation of $W(3,8)$ is a point stabiliser in $\mathrm{P\Gamma Sp}(4,8)$. 
This group therefore has the form $2^9 : (7 \cdot \PSL(2,8) \cdot 3)$, and hence, by divisibilty, cannot contain $A_{10}$ as a subgroup.

Finally, suppose that $nf=4$. 
By \cite[Theorem 6]{Penttila}, there are precisely three pseudo-ovals of $\PG(11,2)$ and each one is elementary: a pseudo-conic, a pseudo-pointed conic, and the field-reduced image of the Lunelli--Sce oval of $\PG(2,16)$. 
The first two examples complete to the pseudo-hyperconic, and the third completes to the field-reduced image of the Lunelli--Sce hyperoval. 
By Thas' result, these are the only pseudo-hyperovals of $\PG(11,2)$. 
By Korchmaros' result, the Lunelli--Sce hyperoval is transitive, and hence its field-reduced image is a transitive pseudo-hyperoval. 
It remains to show that the pseudo-hyperconic is intransitive. 
The stabiliser in $\PGL(2,16)$ of the hyperconic contains $\PSL(2,16)$ and has precisely two orbits: a conic (17 points) and its nucleus (one point). 
If the stabiliser in $\PGL(12,2)$ of the pseudo-hyperconic acted transitively on the pseudo-hyperconic, then it would act $4$-transitively, and hence primitively (as $\PSL(2,16)$ acts $3$-transitively on the hyperconic).
The only $4$-transitive groups of degree $18$ are $A_{18}$ and $S_{18}$ (see \cite{Sims}), but $A_{18}$ is not a subgroup of $\PGL(12,2)$ (see \cite{WagnerEven}). 
\end{proof}

\begin{remark} 
Instead of citing Penttila's unpublished manuscript \cite{Penttila}, we could prove Proposition \ref{propn:elementary} by using the results contained in 
Steinke's 2006 paper \cite{Steinke2006} and the references therein. 
By Steinke's work, and \cite[Theorem 2.4]{RVdV}, any pseudo-oval of $\PG(3m-1,2)$, $m\le 4$, is elementary. 
The hyperovals of $\PG(2,2^m)$, $m\le 4$, have been classified \cite{Hall216} and they are (up to equivalence) hyperconics or the Lunelli--Sce hyperoval of $\PG(2,16)$. 
If $m> 2$, the stabiliser of the hyperconic of $\PG(2,2^m)$ has two orbits on the elements of the hyperconic (that is, the conic and its nucleus), whilst the stabilisers of the hyperconic of $\PG(2,4)$ and the Lunelli--Sce hyperoval are transitive. 
So for example, if we consider the case $nf=4$, there are just three ovals of $\PG(2,16)$: the conic, the pointed conic, and the Lunelli--Sce oval. 
So by Thas' theorem \cite[\S4.10]{Thas71}, we have three pseudo-ovals of $\PG(11,2)$, and we proceed as we did in the proof above.
\end{remark}

\section{Completing the proof of Theorem~\ref{mainthm}: the case \texorpdfstring{$nf>4$}{}} \label{sec4}

We now prove that, for $nf>4$, there is no irreducible subgroup of $\GL(3n,2^f)$ that acts transitively on a set of size $2^{nf}+2$. 
This implies that there is no pseudo-hyperoval of $\PG(3n-1,2^f)$ with an irreducible transitive stabiliser for $nf>4$, and, together with Proposition~\ref{propn:elementary}, proves Theorem~\ref{mainthm}. 
For convenience, let us write $d = 3nf$ (as in the proof of Proposition~\ref{pseudohyperoval}) and apply field reduction. 
That is, we show that, for $d>12$ with $d$ divisible by $3$, no irreducible group $G \le \GL(d,2)$ admits a transitive permutation representation of degree $2^{d/3}+2$. 

Recall that a {\em primitive prime divisor} of $q^e-1$, for $q$ a prime power and $e$ a positive integer, is a prime that divides $q^e-1$ and does not divide $q^i-1$ for any $i<e$. 
Now take $e = 2d/3-2$ and observe that $2^{d/3}+2 = 2(2^{e/2}+1)$. 
Since $|G|$ is divisible by $2^{d/3}+2$, it is divisible by $2^{e/2}+1$, and hence by every primitive prime divisor of $(2^{e/2})^2-1 = 2^e-1$. 
Thus $|G|$ is divisible by the product $\Phi_e^*(q)$ of all the primitive prime divisors of $q^e-1$ (including multiplicities), where for us $q=2$. 
Since $d>12$, the exponent $e=2d/3-2$ satisfies $e>d/2$, and hence we are able to apply Bamberg and Penttila's refinement \cite[Theorem 3.1]{Bamberg:2008rr} of a result of Guralnick et al.~\cite[Main Theorem]{Guralnick:1999th} to determine the irreducible subgroups $G \le \GL(d,2)$ with order divisible by $\Phi_e^*(2)$. 
We then check whether the groups obtained from \cite[Theorem 3.1]{Bamberg:2008rr} admit a transitive permutation representation of degree $2^{d/3}+2$.

\subsection{Classical, imprimitive and symplectic type examples}

In the classical examples case, $G$ must contain one of the following as a normal subgroup: $\SL(d,2)$, $\Sp(d,2)$ or, for $d$ even, $\Omega^\pm(d,2)$. 
(The unitary case in \cite[Theorem 3.1]{Bamberg:2008rr} is excluded because $e$ must be odd while $e=2d/3-2$ is even, and the odd-dimensional orthogonal case is excluded because $dq$ must be odd while $q=2$.) 
However, if such a group $G$ acts transitively on $2^{d/3}+2$ points, then by \cite[Lemma~7.1]{Bamberg:2008rr} we must have $2^{d/3}+2 > (2^{d/2}+1)(2^{d/2-1}-1)$. 
This is impossible because $2^{d/2}+1\ge 2^{d/3}+2$ for all $d$ and $2^{d/2-1}-1 \ge 1$ for $d/2-1\ge 1$, namely for $d\ge 4$ and hence for $d>12$. 

In the imprimitive examples case, by \cite[Theorem 3.1]{Bamberg:2008rr}, $G \le \GL(1,2) \wr S_d = S_d$, $G$ acts primitively of degree $d$ and the only admissible values of $(e,d)$ with $d=3e/2+3$ and $d>12$ are $(10,18)$, $(12,21)$, $(18,30)$. 
For $d=18$, $21$, $30$, $|G|$ is divisible by $2^{d/3}+2=66$, $130$, $1026$, and hence by $11$, $13$, $17$, respectively. 
It follows from Jordan's Theorem \cite[Theorem 13.9]{Wielandt} that $G=A_d$ or $S_d$. 
However, we have assumed that $G$ acts transitively on $2^{d/3}+2$ points, so $G$ must have a subgroup of index $2^{d/3}+2$. 
Therefore, $G \neq A_d$ or $S_d$ because by \cite[Theorem 5.2A]{DM}, $A_d$ has no subgroup of index $2^{d/3}+2$ or $(2^{d/3}+2)/2$. 
(The three cases of \cite[Theorem 5.2A]{DM} do not occur: in case~(i), $r=1$ or $2$; in case~(ii), the index is greater than $1026$; and in case~(iii), $d$ is less than $18$.)

Symplectic type examples do not arise for $G \le \GL(d,q)$ when $q=2$: the only examples listed in \cite[Theorem~3.1]{Bamberg:2008rr} are in characteristic $3$ or $5$.

\subsection{Nearly simple examples}

In this case, $G$ is absolutely irreducible and the socle of $G$ is a nonabelian simple group. 
By \cite[Theorem 3.1]{Bamberg:2008rr}, the only sporadic example with $d>12$ has the socle of $G$ equal to $J_1$ and $(e,d)=(18,20)$, but in this case $e\neq 2d/3-2$. 
Similarly, no cross-characteristic examples or natural-characteristic examples satisfy both $d>12$ and $e=2d/3-2$. 

It remains to check the case where the socle of $G$ is an alternating group $A_m$ with $m\ge 5$. 
The only potential examples with $d>12$ are the deleted permutation module examples (all other alternating group examples have $d\le 8 < 12$). 
Here the characteristic is $2$ and $A_m \le G \le S_m \times Z = S_m$, where 
\[
m = \begin{cases}
d+1 & \text{if } 2 \text{ does not divide } m \\
d+2 & \text{if } 2 \text{ divides } m.
\end{cases}
\]
The only possibilities with $q=2$ and $e = 2d/3-2 > 6$ are for $e=10,12,18$. 
The case $e=12$, for which $d=21$, is excluded because $d$ is even by the above equation. 
The remaining values $(d,m) = (18,19), (18,20), (30,31), (30,32)$ are excluded upon checking that $A_m$ has no subgroups with index $2^{d/3}+2$ or $(2^{d/3}+2)/2$.

\subsection{Extension field examples}

Here there is a divisor $b \neq 1$ of $\operatorname{gcd}(d,e)$ such that $G$ preserves on $V(d,2)$ a field extension structure of a vector space $V(d/b,2^b)$, and $G \le \GammaL(d/b,2^b)$. 
The examples with $b=d=e+1$ in \cite[Theorem~3.1]{Bamberg:2008rr} do not arise for our $d = 3e/2+3$. 

For the remaining examples, \cite[Theorem~3.1]{Bamberg:2008rr} says that $\Phi^*_e(2)$ is coprime to $b$ and divides the order of $G\cap \GL(d/b,2^b)$, and that $G\cap \GL(d/b,2^b)$ satisfies the hypothesis of \cite[Theorem~3.1]{Bamberg:2008rr} with $d$, $e$ and $q$ replaced by $d/b$, $e/b$ and $q^b$, respectively (and with $q=2$ in our case). 
Hence we can apply \cite[Theorem~3.1]{Bamberg:2008rr} provided that $d/b>2$ and $e/b>2$, and one easily verifies that these conditions are satisfied for $d>12$ and $d=3e/2+3$. 
First observe that the symplectic type examples do not arise in even characteristic, the imprimitive examples do not arise over a field of non-prime order, and the extension field sub-examples do not arise because we may assume that $b$ is maximal.

\subsubsection*{Nearly simple examples}
The alternating group, deleted permutation module examples do not arise because $2^b$ is not prime. 
The other alternating group examples do not arise because
$2^b\kern-0.5pt\not\in\kern-0.5pt \{ 2,3,5,7,9,25 \}$. 
For the sporadic group examples, nominally the only admissible value of  $2^b$ is $4$, but then $(e/b,d/b) = (5,6)$ or $(9,9)$, both contradicting $d=3e/2+3$. 
Similarly, in the cross-characteristic examples the only admissible value of $2^b$ is $4$, and here $e/b=d/b \in \{5,6,9\}$, again contradicting $d=3e/2+3$. 
Finally, in the natural-characteristic examples with $p=2$ we have $(e/b,d/b) = (6,8)$, $(6,6)$ or $(4,4)$, and again none of these satisfy $d=3e/2+3$.

\subsubsection*{Classical examples} \label{sss:ext-class}
Finally, the classical examples are ruled out by applying \cite[Lemma~7.1]{Bamberg:2008rr} as follows. 
Let $H \le G$ be a point stabiliser in the degree $2^{d/3}+2$ action of $G$, so that $|G:H| = 2^{d/3}+2$. 
Write $\hat{G} := G\cap \GL(d/b,2^b)$ and $\hat{H} := H\cap \GL(d/b,2^b)$, and note that
\[
|\hat{G}:\hat{H}| = \frac{|G| |\GL(d/b,q^b)| / |G\cdot\GL(d/b,q^b)|}{|H| |\GL(d/b,q^b)| / |H\cdot\GL(d/b,q^b)|}. 
\]
That is,
\[
|\hat{G}:\hat{H}| = \frac{|G:H|}{x}, \quad \text{where } x := |G\cdot \GL(d/b,2^b) : H\cdot \GL(d/b,2^b)|. 
\]
Hence $\hat{G}$ contains a subgroup, $\hat{H}$, of index $(2^{d/3}+2)/x$ for some $x$. 
We know that $\Phi_e^*(2)$ divides $2^{e/2}+1$.
Moreover, $x$ divides $b$ and $b$ is coprime to $\Phi_e^*(2)$ (by \cite[Theorem~3.1]{Bamberg:2008rr}, as noted above), so $x$ is coprime to $\Phi_e^*(2)$ and hence $\Phi_e^*(2)$ also divides $2(2^{e/2}+1)/x = (2^{d/3}+2)/x$. 
In order to apply \cite[Lemma~7.1]{Bamberg:2008rr}, we must check that $(2^{d/3}+2)/x$ does not divide $2(2^b-1)\operatorname{gcd}(2,2^b-1) = 2(2^b-1)$. 
If it did, then $\Phi_e^*(2)$ would also divide $2(2^b-1)$ and therefore $2^b-1$ (because $\Phi_e^*(2)$ is odd), and this would force $b=e$, which is impossible because $b=\operatorname{gcd}(d,e)$ and $e=2d/3-2>d/2$ (since we assume that $d>12$). 
Therefore, we can apply \cite[Lemma~7.1]{Bamberg:2008rr} with $d$ and $q$ replaced by $d/b$ and $2^b$ to obtain the following bound on $|\hat{G}:\hat{H}| = (2^{d/3}+2)/x$:
\[
\frac{2^{d/3}+2}{x} > \frac{(2^{bd/2}+1)(2^{b(d/2-1)}-1)}{2^b-1}.
\]
This implies, in particular, that
\[
2^{d/3}+2 > \frac{(2^{bd/2}+1)(2^{b(d/2-1)}-1)}{2^b-1},
\]
which is not true for $d>12$, $b\ge 2$. This rules out the classical sub-examples of the extension field case and completes the proof of Theorem~\ref{mainthm}.

\section{Proof of Corollary~\ref{mainCorollary}} \label{corProof}


The hypothesis of Corollary~\ref{mainCorollary} is a special case of the situation described in Proposition~\ref{pseudohyperoval} and Corollary~\ref{pseudohyperovalCorr}. 
That is, we suppose that $\mathcal{Q}$ admits an automorphism group $H$ that acts {\em point-primitively} and has a point-regular abelian normal subgroup $N$. 
If $G$ denotes the group of automorphisms of $N$ induced by the conjugation action of $H$ (as in Section~\ref{prelim}), then $H \cong N \rtimes G$ and we can think of $G$ as the stabiliser of $0$ in the action of $H$ on $N$, so $G$ acts irreducibly on the $3n$-dimensional $K$-vector space $N$ (where $K$, $n$ are as in Proposition~\ref{pseudohyperoval}). 
Moreover, $H$ is assumed to act transitively on the lines of $\mathcal{Q}$, so by Theorem~\ref{transitivepseudohyperoval}, $G$ transitively permutes the $n$-dimensional $K$-vector subspaces comprising the pseudo-hyperoval $\mathcal{O} = \{U_1,\ldots,U_{t+1}\}$. 
Now, $\mathcal{Q}$ can be identified with $T_{3n-1,n-1}^*(\mathcal{O})$: the points of $\mathcal{Q}$ correspond to the vectors of $N$, and the lines of $\mathcal{Q}$ to the right cosets of the $U_i$ \cite[Lemma 1]{HM61}. 
Hence, by Corollary~\ref{cor:Tstar}, $\mathcal{Q} \cong T_2^*(\mathcal{H})$, where $\mathcal{H}$ is a hyperoval from which $\mathcal{O}$ is obtained via field reduction. 
By Theorem~\ref{mainthm}, $\mathcal{H}$ is either the regular hyperoval of $\PG(2,2)$, the regular hyperoval of $\PG(2,4)$, or the Lunelli--Sce hyperoval of $\PG(2,16)$. 
Corollary~\ref{mainCorollary} follows upon noting that the first of these examples is excluded as then $T_2^*(\mathcal{H})$ would have order $(1,3)$, contradicting the assumption that $\mathcal{Q}$ is thick.

\appendix
\section{Data for Remark~\ref{eg}} \label{app}

The following data were determined using the computer algebra system {\sf GAP} \cite{gap}. 
Let $\mathcal{O}$ be the pseudo-hyperoval of $\PG(11,2)$ obtained by field reduction of the Lunelli--Sce hyperoval of $\PG(2,16)$. 
The setwise stabiliser of $\mathcal{O}$ is $G = \langle a,b \rangle$, where the generators are given as $12\times 12$ matrices by
\[
a :=\kern-2pt \left[ \begin{array}{llllllllllll}
1&0&1&1&0&1&1&1&0&1&0&0\\
0&0&1&1&1&1&0&0&0&1&0&1\\
1&0&0&1&1&1&1&1&0&0&1&0\\
1&1&0&1&0&1&1&0&1&1&1&0\\
0&0&0&0&1&0&1&1&0&0&0&0\\
0&0&0&0&0&0&1&1&0&0&0&0\\
0&0&0&0&1&0&0&1&0&0&0&0\\
0&0&0&0&1&1&0&1&0&0&0&0\\
0&0&0&0&0&0&0&0&1&0&1&1\\
0&0&0&0&0&0&0&0&0&0&1&1\\
0&0&0&0&0&0&0&0&1&0&0&1\\
0&0&0&0&0&0&0&0&1&1&0&1
\end{array} \right]\kern-2pt,
b :=\kern-2pt \left[ \begin{array}{llllllllllll}
1&0&0&0&1&1&0&0&1&1&1&0\\[-1mm]
0&1&0&0&0&1&1&0&0&1&1&1\\
0&0&1&0&0&0&1&1&1&1&1&1\\
0&0&0&1&1&1&0&1&1&0&1&1\\
0&1&0&1&1&0&0&0&1&0&1&0\\
1&1&1&0&0&1&0&0&0&1&0&1\\
0&1&1&1&0&0&1&0&1&1&1&0\\
1&1&1&1&0&0&0&1&0&1&1&1\\
1&0&0&0&1&1&1&0&0&0&1&0\\
0&1&0&0&0&1&1&1&0&0&0&1\\
0&0&1&0&1&1&1&1&1&1&0&0\\
0&0&0&1&1&0&1&1&0&1&1&0
\end{array} \right].
\]

A representative element of $\mathcal{O}$ is the row space of the $4\times 12$ matrix $[{\bf 1}_{4\times 4} \; {\bf 0}_{4\times 4} \; {\bf 0}_{4\times 4}]$. 
The setwise stabiliser $G=\langle a,b\rangle$ has order $2160=2^4\cdot3^3\cdot5$. 
It has a normal extraspecial subgroup~$M$ of order~$27$ and exponent~$3$. 
A power--conjugate presentation for $G/M$ is
\[
G/M=\langle x,y,z\mid x^2=y^8=z^5=1,\ y^x=y,\ z^x=z^{-1},\ z^y=z^3\rangle 
\cong C_5 \rtimes (C_2\times C_8),
\]
where $x=aba^5b^2M$, $y=aM$, and $z=b^{12}M$.
The natural module $\F_2^{12}$, viewed as an $\F_2 M$-module, equals $U\oplus U$ where $U$ is irreducible but not absolutely irreducible, and ${\rm End}_{\F_2M}(U)=\F_4$.
Thus there are precisely $|\F_4|+1=5$ subspaces $U_1,\dots,U_5$ of $\F_2^{12}$ that are $6$-dimensional and fixed by~$M$.
Using {\sf GAP} we readily see that $G$ has precisely~$14$ subgroups that are transitive on $\mathcal{O}$, and seven of these are reducible on $\F_2^{12}$ (fixing one or five of the $U_i$).
Generators for the reducible transitive subgroups of $G$, and their orders, are listed in Table~\ref{tab1}.

\renewcommand{\arraystretch}{1.2}  
\begin{table}[!t]
\caption{Reducible transitive subgroups of $G = \langle a,b \rangle$ in Appendix~\ref{app}, where $c:=b^5$.}
\begin{tabular}{|l|l|}\hline
Order & Generators for reducible transitive subgroups \\
\hline
54 &  $\langle a^2c^3, a^2c^2a^2, (aca)^2 \rangle$,\quad $\langle a^2c, (aca)^2 \rangle$ \\
108 & $\langle c, (aca)^2  \rangle$,\hskip22mm $\langle c^2 ,a^2c, (aca)^2  \rangle$ \\
216 & $\langle a^2, a c  \rangle$,\hskip27mm $\langle a^2, c \rangle$ \\
432 & $\langle a, c  \rangle$ \\ \hline
\end{tabular}
\label{tab1}
\end{table}

\section*{Acknowledgments} 

We would like to thank Geertrui Van de Voorde and Tim Penttila for making preprints of their work available to us, and we are especially grateful to the former for spotting an error in an earlier draft. 
The first author acknowledges the support of the Australian Research Council Future Fellowship FT120100036.
The second author acknowledges the support of the Australian Research Council Discovery Grant DP130100106. 
The second, third and fourth authors acknowledge the support of and the Australian Research Council Discovery Grant DP1401000416.

\end{document}